\documentclass[12pt]{amsart}
\usepackage{amssymb,latexsym}
\usepackage[english]{babel}
\usepackage[pdftex,colorlinks=true]{hyperref}
\usepackage[utf8]{inputenc}
\usepackage[T1]{fontenc}
\usepackage{XCharter}\usepackage{eulervm}
\usepackage[shortlabels]{enumitem}
\usepackage{graphicx}

\theoremstyle{plain}
\newtheorem{thm}{Theorem}[section]
\newtheorem*{conj}{Conjecture}
\newtheorem{lem}[thm]{Lemma}
\newtheorem{prop}[thm]{Proposition}
\newtheorem{cor}[thm]{Corollary}

\theoremstyle{definition}
\newtheorem*{quest}{Open problem} 

\theoremstyle{remark}
\newtheorem{rem}[thm]{Remark}

\numberwithin{equation}{section}

\newcommand{\R}{\mathbb{R}}

\newcommand{\Z}{\mathbb{Z}}

\newcommand{\Vk}{\mathcal{V}_k^\nabla(\R^n)}
\newcommand{\Pk}{\mathcal{P}_k^\nabla(\R^n)}
\newcommand{\Prop}{\mathcal{P}^\nabla}
\newcommand{\V}{\mathcal{V}^\nabla}
\newcommand{\Vsq}[2]{\mathcal{V}_{#1}^\nabla[\R^{#2}]}
\newcommand{\Psq}[2]{\mathcal{P}_{#1}^\nabla[\R^{#2}]}
\newcommand{\restrictionmap}[2]{{#1}\mathpunct\restriction\hbox{}_{#2}}
\providecommand{\abs}[1]{\left\lvert#1\right\rvert}
\newcommand{\wt}{\widetilde}
\newcommand{\wh}{\widehat}

\DeclareMathOperator{\id}{id}
\DeclareMathOperator{\rank}{rank}
\DeclareMathOperator{\Hess}{Hess}
\DeclareMathOperator{\sign}{sign}

\DeclareMathOperator{\Ind}{Ind}
\DeclareMathOperator{\cl}{cl}
\DeclareMathOperator{\arc}{arc}

\title{Gradient versus proper gradient homotopies}
\author[P. Bart{\l}omiejczyk]
{Piotr Bart{\l}omiejczyk}

\address{Faculty of Applied Physics and Mathematics,
Gda{\'n}sk University of Technology,
Gabriela Narutowicza 11/12,
80-233 Gda{\'{n}}sk, Poland}

\email{piobartl@pg.edu.pl}

\author[P. Nowak-Przygodzki]
{Piotr Nowak-Przygodzki}

\address{Sopot, Poland}

\email{piotrnp@wp.pl}

\date{\today}

\begin{document}

\begin{abstract}
We compare the sets of homotopy classes
of gradient and proper gradient vector fields in the plane.
Namely, we show that gradient and proper gradient 
homotopy classifications are essentially different. 
We provide a complete description of the sets of 
homotopy classes of gradient maps from $\R^n$ to $\R^n$
and proper gradient maps from $\R^2$ to $\R^2$
with the Brouwer degree greater or equal to zero.
\end{abstract}

\maketitle


\section*{Introduction}\label{sec:intro}
The search for new homotopy invariants in the class of gradient maps
has a long history. In 1985, to obtain new bifurcation results, 
E.~N.~Dancer \cite{D} introduced a new topological invariant
for $S^1$\babelhyphen{hard}equivariant gradient maps. 
In turn, A.~Parusi{\'n}ski \cite{P} showed that if two gradient vector fields
on the unit disc $D^n$ nonvanishing on $S^{n-1}$ are homotopic,
i.e., have the same Brouwer degree,
then they are also gradient homotopic.
Similarly, the authors of this paper proved in \cite{BP1,BP2}
that there is no better invariant than the Brouwer degree
for gradient and proper gradient otopies in $\R^n$.
Recall that otopy was introduced by J.~C.~Becker and D.~H.~Gottlieb \cite{BG}
as a very useful generalization of the concept of homotopy.

However, quite surprisingly, M.~Starostka \cite{S} showed
that for $n\ge2$ there exist proper gradient vector fields in $\R^n$
which are homotopic but not proper gradient homotopic.
Roughly speaking, he proved that the identity and the minus identity on the plane
are not proper gradient homotopic and then generalized this result to $\R^n$.
Since his reasoning is nice and elegant, we will present it briefly here.
First recall that the linear source and sink in the plane are isolated
invariant sets with different homological Conley indices.
Namely, since $(D_r(0),\partial D_r(0))$ and $(D_r(0),\emptyset)$,
where $D_r(0)$ denotes the $r$-disc at the origin,
are index pairs for the source and sink respectively,
the respective homological Conley indices are
(see Figure~\ref{fig:conley})
\[
\text{CH}_*(D_r(0),\eta_{\id})=H_*(S^2,\text{pt})
\,\,\text{ and }\,\,
\text{CH}_*(D_r(0),\eta_{-\id})=H_*(S^0,\text{pt}),
\]
where $\eta_f$ denotes the flow generated by the vector field $f$.
Suppose now that there is a proper gradient homotopy connecting $\id$ to $-\id$.
Such a homotopy determines a continuation between the gradient flows $\eta_{\id}$ 
and $\eta_{-\id}$ for which a sufficiently large disc $D_r(0)$
is a common isolating neighbourhood for all parameter values of the continuation
(this is true for proper gradient vector fields and homotopies).
Thus, by the continuation property of the Conley index, 
$\text{CH}_*(D_r(0),\eta_{\id})=\text{CH}_*(D_r(0),\eta_{-\id})$, a contradiction.
To summarize, if we restrict ourselves to proper gradient vector fields and homotopies,
then the Conley index is a better invariant than the Brouwer degree.
\begin{figure}[ht]
\centering
\includegraphics[scale=1.4,trim= 62mm 218mm 70mm 45mm]{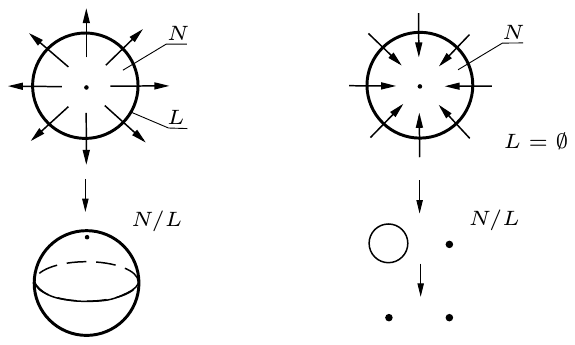}
\caption{Conley indices of a source and sink}
\label{fig:conley}
\end{figure}

In this paper we strengthen and complement Starostka's result.
We present the comparison of two homotopy classifications
of gradient vector fields in the plane: gradient and proper gradient.
Namely, we show that the set of homotopy classes of gradient
vector fields in $\R^n$ having the same Brouwer degree 
is a singleton (a Parusi{\'n}ski-type theorem).
On the other hand, the set of homotopy classes of proper gradient
vector fields in $\R^2$ the same Brouwer degree is empty 
if the degree is greater than $1$, has exactly two elements
if the degree is equal to $1$ and has one element 
if the degree is equal to $0$.
What is still lacking is a description of this set
for the degree less than $0$.  It also would be desirable to
provide the proper gradient homotopy classification 
for the general case of $\R^n$.

The organization of the paper is as follows.
Section~\ref{sec:prel} contains some preliminaries. 
Our main four theorems are stated in Section~\ref{sec:main}.
These results are proved subsequently 
in Sections~\ref{sec:proof1}-\ref{sec:proof4}.
Finally, Appendix~\ref{sec:appA} presents a series
of technical results needed in previous sections.


\section{Preliminaries}
\label{sec:prel}

In what follows, a map denotes always a continuous function
and $\deg$ denotes the classical Brouwer degree.
\subsection{Gradient and proper gradient maps}
Recall that a map $f$ is called \emph{gradient} 
if there is a~$C^1$~function
$\varphi\colon\R^n\to\R$ such that $f=\nabla\varphi$
and is called \emph{proper} if preimages of 
compact sets for $f$ are compact.
Let $I:=[0,1]$.

We write $f\in\V(\R^n)$ ($f\in\Prop(\R^n)$) if
\begin{enumerate}
	\item $f$ is gradient, 
	\item $f^{-1}(0)$ is compact ($f$ is proper).
\end{enumerate}
Moreover, let $\Vk:=\{f\in\V(\R^n)\mid\deg f=k\}$ and
$\Pk:=\{f\in\Prop(\R^n)\mid\deg f=k\}$. 

\subsection{Gradient and proper gradient homotopies}
Apart from maps we consider two classes of homotopies:
gradient and proper gradient.
Namely, a map $h\colon I\times\R^n\to\R^n$ is called
a \emph{(proper) gradient homotopy} if
\begin{enumerate}
	\item $h_t(\cdot):=h(t,\cdot)$ is gradient for each $t\in I$,
	\item $h^{-1}(0)$ is compact ($h$ is proper).
\end{enumerate}
If $h$ is a (proper) gradient homotopy, 
we say that $h_0$ and $h_1$ are \emph{(proper) gradient homotopic}. 
The relation of being (proper) gradient homotopic 
is an equivalence relation in $\Vk$ ($\Pk$).
The sets of homotopy classes of the respective relation
will be denoted by $\Vsq{k}{n}$ and $\Psq{k}{n}$.

\subsection{Hessian maps}
Let us consider a $C^2$ function $\varphi\colon\R^2\to\R$. 
Assume that $p\in\R^2$ is a nondegenerate critical point of $\varphi$. 
Let $\Hess_p\!\varphi$ denote the Hessian of $\varphi$ at $p$. 
In that situation, the Hessian is nondegenerate bilinear symmetric form and, 
in consequence, its matrix is invertible symmetric.  
Let us make two simple observations.
\begin{lem}\label{lem:symm}
Any two elements of the space of invertible symmetric matrices
are in the same path-connected component
if and only if they have the same signature.
\end{lem}

\begin{cor}\label{cor:symm}
Let $\varphi\colon\R^2\to\R$ and $p$ 
be a nondegenerate critical point of $\varphi$. Then the Hessian map
$\Hess_p\!\varphi\colon\R^2\to\R^2$ is proper gradient homotopic
to $\id_{\R^2}$ if $\sign\Hess_p\!\varphi=0$ or
to $-\id_{\R^2}$ if $\sign\Hess_p\!\varphi=2$.
\end{cor}

\subsection{Local flows}
Let $\Omega$ be an open subset of $R^2$.
A map $\eta\colon A\to\Omega$ is called 
a \emph{local flow} on $\Omega$ if:
\begin{itemize}
	\item $A$ is an open neighbourhood of 
	$\{0\}\times\Omega$ in $\R\times\Omega$,
	\item for each $x\in\Omega$ there are 
	$\alpha_x, \omega_x\in\R\cup\{\pm\infty\}$
	such that 
	$(\alpha_x,\omega_x)$\\ $=\{t\in\R\mid (t,x)\in A\}$,
	\item $\eta(0,x)=x$ and $\eta(s,\eta(t,x))=\eta(s+t,x)$ 
	for all $x\in\Omega$ and $s,t\in(\alpha_x,\omega_x)$ 
	such that $s+t\in(\alpha_x,\omega_x)$.
\end{itemize}
Assume that $f\colon\Omega\to\R^2$ is a $C^1$ vector field.
It is well-known that
if $t\to\eta(t,x_0)$ is a solution of the initial value problem
\[
\dot{x}=f(x),\quad x(0)=x_0
\] 
and $(\alpha_{x_0},\omega_{x_0})$ is the maximal interval of
existence of the solution of the initial value problem, 
then the map $\eta$ is a local flow on $\Omega$. 

\begin{prop}[{\cite[Ch. 6]{H}}]\label{prop:generic}
Any element of $\Prop(\R^2)$ is proper gradient homotopic to a generic map.
\end{prop}

\subsection{Notation}
Let us denote by $B_r(p)$ ($D_r(p)$) the open (closed)
$r$-ball in $\R^n$ around~$p$.

\section{Main results} 
\label{sec:main}

Let us formulate the main results of our paper.
 
\begin{thm}\label{thm:one}
$\Vsq{k}{n}$ is a singleton for each $k\in\Z$ and $n\ge2$.
\end{thm}

\begin{thm}\label{thm:twoone}
$\Psq{k}{2}$ is empty for $k>1$.
\end{thm}

\begin{thm}\label{thm:twotwo}
$\Psq{0}{2}$ is a singleton.
\end{thm}

\begin{thm}\label{thm:twothree}
$\Psq{1}{2}$ has at most two elements.
\end{thm}

Combining Theorem \ref{thm:twothree} with the theorem of M. Starostka 
(see \cite[Main Theorem]{S}) gives immediately
the following result.
\begin{cor}
$\Psq{1}{2}$ has exactly two elements: the class of
the identity and the minus identity.
\end{cor}

We close this section with the following conjecture 
and open problem.

\begin{conj}
$\Psq{k}{2}$ is a singleton for $k<0$.
\end{conj}

\begin{quest}
Give the description of the set $\Psq{k}{n}$
for any $k\in\Z$ and $n>2$.
\end{quest}

\section{Proof of Theorem \texorpdfstring{\ref{thm:one}}{2.1}}
\label{sec:proof1}

A~slight modification of the reasoning presented in the proof of Lemma~4
in~\cite{P} shows that the sets $\Vsq{k}{n}$ are nonempty.  
Now we prove that $\Vsq{k}{n}$ consist of only one element. 
Let $\nabla\varphi,\nabla\psi\in\V_k(\R^n)$.
There is $r>0$ such that 
$(\nabla\varphi)^{-1}(0)\cup(\nabla\psi)^{-1}(0)\Subset B_r(0)$.
By the Parusi\'nski theorem (\cite[Thm 1]{P}), there is a $C^1$ function
$\zeta\colon I\times D_r(0)\to\R$ such that
\begin{itemize}
	\item $\nabla_x\zeta(t,x)\neq0$ for all $t\in I$ and $x\in\partial D_r(0)$,
	\item $\nabla\zeta_0=\nabla\big(\restrictionmap{\varphi}{D_r(0)}\big)$,
	\item $\nabla\zeta_1=\nabla\big(\restrictionmap{\psi}{D_r(0)}\big)$.
\end{itemize}
Assume that $\theta$ is a diffeotopy from Lemma~\ref{lem:proof1a}.
Let us define three homotopies $h^i\colon I\times\R^n\to\R^n$ ($i=1,2,3$)
by the formulas
\begin{align*}
h_t^1(t,x)&=\nabla_x\varphi(\theta(t,x)),\\
h_t^2(t,x)&=\nabla_x\psi(\theta(t,x)),\\
h_t^3(t,x)&=\nabla_x(\zeta(t,\theta_1(x))).
\end{align*}
By Lemma~\ref{lem:proof1b}, $h^1$ and $h^2$ are gradient homotopies
and by Lemma~\ref{lem:proof1c}, $h^3$ is a gradient homotopy.
Thus we obtain the following sequence of the gradient homotopy relations
\[
\nabla\varphi=h_0^1\sim h_1^1=h_0^3\sim
h_1^3=h_1^2\sim h_0^2=\nabla\psi,
\]
which completes the proof.\qed


\section{Proof of Theorem \texorpdfstring{\ref{thm:twothree}}{2.4}} 
\label{sec:proof2}
The following two propositions are crucial for the proof of
Theorem~\ref{thm:twothree}. Assume that $f=\nabla\varphi$ is generic.
\begin{prop}\label{prop:point}
Let $f^{-1}(0)=\{p\}$. If $p$ is a source then $f\sim\id_{\R^2}$.
If $p$ is a sink then $f\sim-\id_{\R^2}$.
\end{prop}

\begin{proof}
Assume that $p$ is a source. By Corollary~\ref{cor:proof2}, 
$f$ is proper gradient homotopic to the Hessian map 
$\Hess_p\!\varphi\colon\R^2\to\R^2$ and by Corollary~\ref{cor:symm}, 
$\Hess_p\!\varphi$ is proper gradient homotopic to $\id_{\R^2}$.
The same reasoning applies to the case of a sink.
\end{proof}

Let $A_f^-$ ($A_f^+$) denote the set of sources (sinks)
of $f$, $A_f=A_f^-\cup A_f^+$ and $B_f$ the set of saddles.
\begin{prop}\label{prop:cancel}
If $A_f$ and $B_f$ are nonempty then there is a generic map $f'$ such that 
$f\sim f'$, $\abs{A_{f'}}<\abs{A_{f}}$ and  $\abs{B_{f'}}<\abs{B_{f}}$.
\end{prop}

The proof of Proposition \ref{prop:cancel} 
will be preceded by a series of lemmas. 
Let us start with the following notation.
Assume that $x\in A_f^-$ and $y\in A_f^+\cup B_f\cup\{\infty\}$. Set
\begin{align*}
E_x^-&=\{z\in\R^2\mid\omega^-(z)=x\},\\
E_y^+&=\{z\in\R^2\mid\omega^+(z)=y\}.
\end{align*}
From now on, $\eta(t,z)=\eta^t(z)$ denotes the local flow generated by $f$.
Observe that for $z\in E_x^-$ the vector
\[
v_z=\lim_{t\to-\infty}\frac{f(\eta^t(z))}{\abs{f(\eta^t(z))}}
\]
is well-defined. Finally, let us denote by $V_x^y$ the set
$\{v_z\mid z\in E_x^-\cap E_y^+\}$.

The following lemma describes properties of sets $V_x^y$.
\begin{lem}\label{lem:open}
Assume that $y\in A_f^+\cup\{\infty\}$. Then
\begin{enumerate}
	\item $V_x^y$ is an open subset of $S^1$,
	\item if $V_x^y=S^1$ then $x\in A_f^-$ is the only 
	stationary point of $f$ and $y=\infty$.
\end{enumerate}
\end{lem}

\begin{proof}
Note that there is a neighbourhood
$U^-$ of $x$ such that $U^-$ is the unit ball and $f=\id$ on $U^-$
in some coordinate system. Now we can identify the set of directions 
$\{v_z\in S^1\}$ with $\partial U^-$.\vspace{1mm}

\noindent\emph{Ad (1).} Assume that $y\in A_f^+$.
Analogously as for $x$, there is a neighbourhood
$U^+$ of $y$ such that $U^+$ is a ball and $f=-\id$ on $U^+$.
Let $z_0\in V_x^y\subset S^1$.
There is $T\in\R$ such that $\eta^T(z_0)\in U^+$. Therefore we
can choose a neighbourhood $W$ of $z_0$ in $S^1$ such that
$\eta^T(z)\in U^+$ for all $z\in W$. Hence $W\subset V_x^y$,
and in consequence, $V_x^y$ is open.

Now let $y=\infty$. Since $f$ is proper, there is $\rho>0$ 
such that for $z\not\in D_\rho(0)$
$\lim_{t\to-\infty}\eta^t(z)=x$ implies $\lim_{t\to\infty}\eta^t(z)=\infty$
(see \cite[Prop~2.4]{S}). 
Let $z_0\in V_x^\infty\subset S^1=\partial U^-$.
Then there is $T>0$ such that $\eta^T(z_0)\not\in D_\rho(0)$.
Therefore there exists a neighbourhood $W$ of $z_0$ in $S^1$ such that
$\eta^T(z)\not\in D_\rho(0)$ for all $z\in W$,
which proves that $V_x^\infty$ is open.\vspace{1mm}

\noindent\emph{Ad (2).} 
Without loss of generality we can assume that $\varphi(x)=0$
and $\varphi(z)=1$ for $z\in\partial U^-$.
Set 
\[
\beta_z:=\sup{\{\varphi(\eta^t(z))\mid t\in[0,\omega_z)\}}
\]
for $z\in\partial U^-$ and 
$\beta:=\inf{\{\beta_z\mid z\in\partial U^-\}}$.
Since $\varphi$ increases on trajectories of $\eta$,
for $z\in\partial U^-$ and $\alpha\in(0,\beta)$
there is a unique $t(z,\alpha)\in(-\infty,\omega_z)$ 
such that $\varphi(\eta^{t(z,\alpha)}(z))=\alpha$.
Write 
\[
S(\alpha)=\{\eta^{t(z,\alpha)}(z)\mid z\in\partial U^-\}.
\]
Note that $S(\alpha)$ is homeomorphic to $S^1$
for $\alpha\in(0,\beta)$ and $S(1)=\partial U^-$.

We begin by proving that $y=\infty$.
Conversely, suppose that $y$ is a point.
Note that $y$ cannot be a saddle, because for a saddle there are
only two directions of approach along the flow.
Hence $y$ is a sink and $\varphi(y)=\beta$.
Once again, let $U^+$ be a disc neighbourhood of $y$ 
such that $f=-\id$ on $U^+$.
By compactness of $\partial U^-$,
while $\alpha$ is approaching $\beta$,
a level set $S(\alpha)$ is closer to $y$.
From this observation it follows that there is
an embedding $\gamma\colon S^2\to\R^2$,
which maps poles on $x$ and $y$, a contradiction.
This clearly forces that  $y=\infty$.

It remains to prove that $A_f\cup B_f=\{x\}$.
There is no loss of generality in assuming that $x=0$.
Since $\nabla\varphi$ is proper, we have $\beta=\infty$.
Therefore $S(\alpha)$ is defined for every $\alpha>0$. 
We show that any $z\neq0$ belongs to a level set
$S(\alpha)$ for some $\alpha>0$ and, in consequence, 
is not a critical point of $\varphi$. For $z\in U^-$
it is obvious. Let $z\not\in U^-$ and 
$\alpha_0=\max{\{\varphi(w)\mid\abs{w}\le\abs{z}\}}$.
Choose arbitrary $\alpha_1>\alpha_0$. Observe that
$\Ind(S(1/2),z)=0$ and 
$\Ind(S(\alpha_1),z)=\Ind(S(\alpha_1),0)=1$,
where $\Ind$ denotes the winding number.
Suppose, contrary to our claim, that 
$z\not\in S(\alpha)$ for $1/2<\alpha<\alpha_1$.
Then $\Ind(S(1/2),z)=\Ind(S(\alpha_1),z)$, a~contradiction.
This completes the proof. 
\end{proof}

The next four lemmas are of utmost importance for
the proof of Proposition~ \ref{prop:cancel}.
\begin{lem}\label{lem:key}
Let $x\in A_f^-$ and $B_f$ is nonempty.
Then there is $y\in B_f$ such that $x$ and $y$
are connected by a trajectory of $\eta$.
\end{lem}

\begin{proof}
Write $V_x^Y:=\cup_{y\in Y}V_x^y$ for 
$Y\subset A_f^+\cup B_f\cup\{\infty\}$.
Since $f$ is generic,
we have
\[
V_x^\infty\cup V_x^{A_f^+}\cup V_x^{B_f}=S^1.
\]
By Lemma~\ref{lem:open}, $V_x^\infty\cup V_x^{A_f^+}$
is an open strict subsets of $S^1$. Hence $V_x^{B_f}\neq\emptyset$,
which is our claim.
\end{proof}

Assume that $x\in A_f^-$. Write
\begin{align*}
A_f^+(x)&:=\{y\in A_f^+\mid V_x^y\neq\emptyset\},\\
B_f(x)&:=\{y\in B_f\mid V_x^y\neq\emptyset\}.
\end{align*}

\begin{lem}\label{lem:height}
If $A_f^+(x)\neq\emptyset$ and $B_f(x)\neq\emptyset$ then 
\[
\min{\{\varphi(y)\mid y\in A_f^+(x)\}}\ge
\min{\{\varphi(y)\mid y\in B_f(x)\}}.
\]
\end{lem}

\begin{proof}
Let $y_1\in A_f^+(x)$ such that 
$\varphi(y_1)=\min{\{\varphi(y)\mid y\in A_f^+(x)\}}$.
By the above and Lemma~\ref{lem:open}(2),  $\emptyset\neq V_x^{y_1}\neq S^1$.
Let $C$ denote a connected component of $V_x^{y_1}$ and let
$a$ be one of the ends of the arc $C$.
By Lemmas \ref{lem:open} and \ref{lem:key},
there is $y_0\in B_f(x)$ such that $a\in V_x^{y_0}$.

Suppose that $\varphi(y_0)>\varphi(y_1)$.
Therefore there are disjoint neighbourhoods $U_0$ of $y_0$ and
$U_1$ of $y_1$ such that for all $z\in U_0$ and $w\in U_1$
we have $\varphi(z)>\varphi(w)$ and, moreover, no trajectory 
leaves $U_1$. Observe that if $a'\in C$
is close enough to $a$ then there is $t_0$ such that
$\eta^{t_0}(a')\in U_0$. Moreover, since $a'\in C$, there is $t_1$
such that $\eta^{t_1}(a')\in U_1$. Note that $t_1>t_0$, because 
for $t\ge t_1$ we have $\eta^{t}(a')\in U_1$. Hence
\[
\varphi\big(\eta^{t_0}(a')\big)<
\varphi\big(\eta^{t_1}(a')\big),
\]
a contradiction. This gives our assertion.
\end{proof}

The following result, which can be found in \cite[Sec. 1]{L}, 
is devoted to the question of cancelling a pair of critical points.

\begin{lem}\label{lem:cancel}
Let us consider a $C^2$ function $\varphi\colon\R^2\to\R$
such that $\nabla\varphi$ is generic and its local flow $\eta$. 
Let $p$ and $q$ be two critical points of $\varphi$ satisfying the following conditions:
\begin{itemize}
\item $W^u(p)$ and $W^s(q)$ intersect transversely
       and the intersection consists of one orbit $l$ of $\eta$,
\item for  some $\epsilon>0$, each orbit of $\eta$ in $W^u(p)$ 
      distinct from $l$ crosses the level
      set $\varphi^{-1}\big(\varphi(q)+\epsilon\big)$.
\end{itemize}
Let $U$ denote an open neighbourhood of
the closure of $W^u(p)\cap \{\varphi\leq\varphi(q)+\epsilon\}$
such that the only critical points in $\cl U$ are $p$ and $q$. 
Then there is a path of smooth functions 
$\left\{\varphi_t\right\}_{ t\in I},$ such that:
\begin{itemize}
	\item $\varphi_0=\varphi$,
	\item for every $t\in I$, $\varphi_t$ coincides 
	with $\varphi$ on $\R^2\setminus U$,
	\item the function $\varphi_t\vert U$ 
	is has two nondegenerate critical points 
	when $0\leq t<1/2$; it has one degenerate 
	critical point when $t=1/2$ 
	and it has no critical points when $1/2<t\leq 1$.
\end{itemize}
\end{lem}

\begin{lem}\label{lem:twoorbits}
Assume that $f=\nabla\varphi$ is generic.
Let $x\in A_f^-$ and $y\in B_f$.
If there are two trajectories of $\eta$ connecting
$x$ to $y$ then there is generic $f'$
otopic to $f$ such that $\abs{A_{f'}}<\abs{A_{f}}$ and 
$\abs{B_{f'}}<\abs{B_{f}}$.
\end{lem}

\begin{proof}
Let us denote by $\Gamma_1$ and $\Gamma_2$ two trajectories
of $\eta$ (smooth curves) connecting $x$ to $y$.
Notice that these trajectories form a straight angle at $y$
(see Figure~\ref{fig:saddle}). Let $G$ stand for the domain bounded 
by $\Gamma_1$ and $\Gamma_2$. Since $x$ is a source,
we can choose points $x_1\in\Gamma_1$, $x_2\in\Gamma_2$
and a level subset $\Pi_x\subset G$ of $\varphi$
connecting them close enough to $x$. Similarly,
since $y$ is a saddle, we can choose 
$y_1\in\Gamma_1$, $y_2\in\Gamma_2$
such that $\varphi(y_1)=\varphi(y_2)$ and a smooth curve
$\Pi_y\subset G$ perpendicular 
to $\Gamma_i$ at $y_i$ for $i=1,2$. 
Let us denote by $G_1$ the domain bounded by
$\Pi_x$, $\Pi_y$ and trajectories connecting $x_i$ to $y_i$.
\begin{figure}[ht]
\centering
\includegraphics[scale=0.7,trim= 70mm 187mm 60mm 45mm]{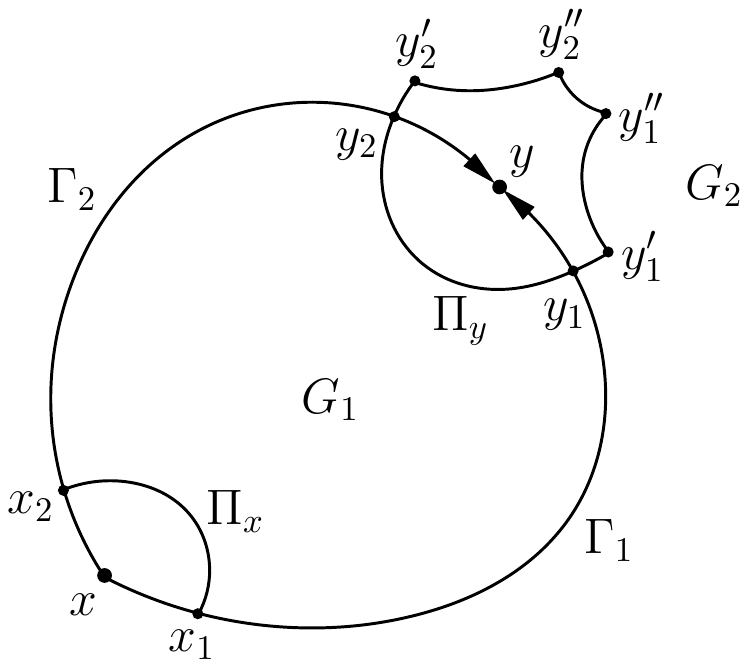}
\caption{Domains $G_1$ and $G_2$}
\label{fig:saddle}
\end{figure}

Let us extend $\Pi_y$ to a little longer smooth curve
$\Pi'_y=\arc y'_1y'_2$ such that segments $\arc y_iy'_i$
are contained in level sets of $\varphi$.
In the neighbourhood of the saddle $y$ choose points
$y''_i$ for $i=1,2$ on the trajectory starting from $y'_i$ 
such that $\varphi(y''_1)=\varphi(y''_2)$.
Points $y''_1$ and $y''_2$ are connected by
a short segment of a level set of $\varphi$.
Let us denote by $G_2$ the domain bounded by
$\Pi'_y$, trajectories connecting $y'_i$ to $y''_i$
and this short segment of a level set
connecting $y''_1$ and $y''_2$.
 
Since $f$ is defined and nonvanishing on
$\partial(G_1\cup G_2)$, we can extend
$\restrictionmap{f}{\partial(G_1\cup G_2)}$
to nonvanishing smooth vector field 
$\wt{f}\colon\partial G_1\cup\partial G_2\to\R^2$
such that $\wt{f}$ is perpendicular to the curve $\Pi_y$.
It is easy to check that $\wt{f}$ satisfies
the assumptions of Lemma~\ref{lem:field} for both
$G_1$ and $G_2$. As a conclusion we obtain that
$\wt{f}$ can be extended to nonvanishing gradient vector field 
$\wh{f}\colon G_1\cup G_2\to\R^2$. Finally, define
$f'\colon\R^2\to\R^2$ by the formula
\[
f'(z)=\begin{cases}
f(z)&\text{if $z\not\in G_1\cup G_2$},\\
\wh{f}(z)&\text{if $z\in G_1\cup G_2$}.
\end{cases}
\]
The map $f'$ satisfies the assertion of Lemma~\ref{lem:twoorbits}.
\end{proof}

\begin{proof}[Proof of Proposition \ref{prop:cancel}]
Without loss of generality we can assume that $A^-_{f}\neq\emptyset$
(in the case $A^+_{f}\neq\emptyset$ the proof is analogous).
Let $x\in A^-_{f}$. By Lemma~\ref{lem:key}, $B_f(x)\neq\emptyset$.
Moreover, by Lemma~\ref{lem:height}, there is $y\in B_f(x)$
which realizes the minimum of 
$F=\{\varphi(z)\mid z\in A^+_{f}(x)\cup B_f(x)\}$.
Applying Lemma~\ref{lem:bump} we can assume that
$y$ is the only minimum in~$F$. There are two possibilities.
There is either one or two trajectories connecting $x$ to $y$.
In the second case  it is enough to apply Lemma~\ref{lem:twoorbits}
to obtain at once the desired conclusion.
Now let us consider the first case.
Observe that all assumptions of Lemma~\ref{lem:cancel} are satisfied.
In consequence, a proper gradient homotopy
allows us to cancel both critical points $x$ and $y$,
which is our assertion.
\end{proof}

It occurs that Theorem \ref{thm:twothree} is now a consequence 
of Propositions~\ref{prop:generic}, \ref{prop:point} and~\ref{prop:cancel} .

\begin{proof}[Proof of Theorem \ref{thm:twothree}]
Let $f\in\Prop_1(\R^2)$. By Proposition~\ref{prop:generic}, 
without loss of generality we can assume that $f$ is generic.
Moreover, $\deg f=\abs{A_{f}}-\abs{B_{f}}=1$.
By~Proposition~\ref{prop:cancel}, there is $f'$ generic such that
$f\sim f'$, $f'^{-1}(0)=\{p\}$ and $p$ is a source or sink.
Hence, by Proposition~\ref{prop:point}, $f\sim\id_{\R^2}$ or $f\sim-\id_{\R^2}$.
\end{proof}

\section{Proof of Theorem \texorpdfstring{\ref{thm:twotwo}}{2.3}}
\label{sec:proof3}

The main result of this section is the following lemma.
\begin{lem}\label{lem:nozero}
If $f,f'\in\Prop(\R^2)$ have no zeroes then $f\sim f'$.
\end{lem}

\begin{proof}
Let $f=\nabla\varphi$. For $t\in[1/2,1]$ write $f_t(x):=f((2-2t)x)$.
Set $c:=\min\{\abs{f(x)}\mid x\in\R^2\}$.
Observe that $c>0$ and for each $t\in[1/2,1]$
\begin{itemize}
  \item $f_t$ are gradient maps,
	\item $\min{\{\abs{f_t(x)}\mid x\in\R^2\}}\ge c$.
\end{itemize}

Next for $t\in[0,1/2]$ put $\xi_t(x):=(1+t\abs{x})x$ and
\[
\Xi_t(f)(x):=\nabla\big(\varphi(\xi_t(x))\big)=
D\xi_t^T(x)\nabla\varphi(\xi_t(x))=
D\xi_t^T(x)f(\xi_t(x)).
\]
Let us check that following inequalities
\begin{enumerate}
	\item $\abs{\xi_t(x)}\ge\abs{x}$,
	\item $\abs{D\xi_t^T(x)(v)}\ge(1+t\abs{x})\abs{v}$,
	\item $\abs{\Xi_t(f)(x)}\ge\abs{f(\xi_t(x))}$,
	\item $\big\lvert\Xi_\frac12(f_t)(x)\big\rvert\ge\big(1+\frac12\abs{x}\big)c$ for $t\in[1/2,1]$.
\end{enumerate}
The first one is obvious. The second follows from the fact that
for a~given $x$ the matrix $D\xi_t(x)$ is diagonal in some basis 
with the elements $(1+2t\abs{x})$ and $(1+t\abs{x})$ on the diagonal.
The third and fourth follow immediately from the second.

Finally, define a homotopy
\[
h_t(x)=\begin{cases}
\Xi_t(f)(x)& \text{if $t\in[0,1/2]$},\\
\Xi_\frac12(f_t)(x)& \text{if $t\in[1/2,1]$}.
\end{cases}
\]
The homotopy $h_t$ is obviously gradient.
Moreover, it is proper. Namely, the first part
of the homotopy is proper from (1) and (3)
and the properness of $f$, and the second part
is proper from (4).

Observe that the homotopy $h_t$ connects $f$ to $\Xi_\frac12(f(0))$,
where $f(0)$ denotes a constant vector field on $\R^2$.
What is left is to show that 
\[
\Xi_\frac12(f(0))\sim\Xi_\frac12(f'(0)).
\]
Note that there is a homotopy $g_t$ between $f(0)$ and $f'(0)$
consisting of nonzero constant vector fields. 
It is immediate that the homotopy $\Xi_\frac12(g_t)$
is proper gradient, which completes the proof.
\end{proof}

\begin{rem}
The last lemma is true for $\Prop(\R^n)$ ($n\ge2$) with the same proof.
\end{rem}

\begin{proof}[Proof of Theorem \ref{thm:twotwo}]
Let $f,f'\in\Prop_0(\R^2)$.
By Proposition~\ref{prop:cancel}, we can assume that $f$ and $f'$
have no zeroes. Lemma~\ref{lem:nozero} now shows that $f\sim f'$.
\end{proof}

\section{Proof of Theorem \texorpdfstring{\ref{thm:twoone}}{2.2}}
\label{sec:proof4}

Let $f\in\Prop_k(\R^2)$ ($k>1$) be generic.
By Proposition~\ref{prop:cancel}, we can assume that
$f$ has no saddles. Observe that
to complete the proof it is enough to show that
if $A^-_f\neq\emptyset$ then $\abs{A^-_f}=\abs{A_f}=1$,
which contradicts our assumption $k>1$
(the case $A^+_f\neq\emptyset$ is analogous).

Let $x\in A^-_f$. Since $\cup_{y\in A^+_f\cup\{\infty\}}V_x^y=S^1$
and, by Lemma~\ref{lem:open}(1), the sets $V_x^y$ are open,
we have $V_x^y=S^1$ for some $y\in A^+_f\cup\{\infty\}$. 
Hence, by Lemma~\ref{lem:open}(2), $y=\infty$ and $x$ is the only 
stationary point of $f$, i.e.,\ $\abs{A^-_f}=\abs{A_f}=1$.\qed

\appendix
\section{} 
\label{sec:appA}
\subsection{Diffeotopies} 
\begin{lem}\label{lem:proof1a}
There is a diffeotopy on the image $\theta\colon I\times\R^n\to\R^n$
such that $\theta_0=\id_{\R^n}$, $\theta_1(\R^n)=B_r(0)$ 
and $B_r(0)\subset\theta_t(\R^n)$ for all $t\in I$.
\end{lem}

\begin{proof}
Let us consider a function $\mu\colon[0,\infty)\to[0,r)$ of the form
$\mu(s)=\frac{2r}{\pi}\arctan s$. Define a straightline homotopy
$\mu_t(s)=(1-t)s+t\mu(s)$. The function $\theta\colon I\times\R^n\to\R^n$
given by $\theta(t,x)=\mu_t(\abs{x})\frac{x}{\abs{x}}$ 
is a diffeotopy with the desired properties.
\end{proof}

\begin{lem}\label{lem:proof1b}
Let $\varphi\colon\R^n\to\R$ be a $C^1$ function,
$\nabla\varphi\in\V(\R^n)$ and 
$\theta\colon I\times\R^n\to\R^n$ be a diffeotopy on the image
such that $(\nabla\varphi)^{-1}(0)\subset\theta_t(\R^n)$ 
for all $t\in I$.
Then $h\colon I\times\R^n\to\R^n$ given by
\[
h(t,x)=\nabla_x(\varphi(\theta (t,x)))
\]
is a gradient homotopy.
\end{lem}

\begin{proof}
It is enough to check that $h^{-1}(0)$ is compact.
Let us define $\wt{\theta}\colon I\times\R^n\to I\times\R^n$
by $\wt{\theta}(t,x)=(t,\theta (t,x))$. Observe that
$\wt{\theta}$ is a homeomorphism on the image and
\[
h^{-1}(0)=
\wt{\theta}^{-1}(I\times(\nabla\varphi)^{-1}(0)),
\]
which proves the lemma.
\end{proof}

\begin{lem}\label{lem:proof1c}
Assume that $\gamma\colon\R^n\to B_r(0)$ is a diffeomorphism and 
$\zeta\colon I\times D_r(0)\to\R$ is continuous and $C^1$
with respect to $x$ such that 
$\nabla_x\zeta(t,x)\neq0$ for all $t\in I$ and $x\in\partial D_r(0)$.
Then the function $h\colon I\times\R^n\to\R^n$ given by
\[
h(t,x)=\nabla_x(\zeta(t,\gamma(x)))
\]
is a gradient homotopy.
\end{lem}

\begin{proof}
The map $\wt{\gamma}\colon I\times\R^n\to I\times B_r(0)$
given by $\wt{\gamma}(t,x)=(t,\gamma(x))$ is a homeomorphism,
the set $(\nabla_x\zeta)^{-1}(0)$ is compact and
\[
h^{-1}(0)=
\wt{\gamma}^{-1}\big((\nabla_x\zeta)^{-1}(0)\big).
\]
Hence the last set is compact.
\end{proof}

\subsection{Milnor's trick}
Let $f\colon\R^2\to\R^2$ be a map differentiable 
at $0$ and $f(0)=0$.
We define a function $h\colon I\times\R^2\to\R^2$
by the formula
\[
h(t,x)=\begin{cases}
\dfrac{f(tx)}{t}& \text{if $t\neq0$},\\
Df_0(x)& \text{if $t=0$}.
\end{cases}
\]

\begin{lem}\label{lem:proof2}
\mbox{}
\begin{enumerate}
	\item The function $h$ is continuous.
	\item If $Df_0$ is nonsingular, $f$ is proper 
	and $f^{-1}(0)=\{0\}$ then $h$ is proper. 
\end{enumerate}
\end{lem}

\begin{proof} \emph{Ad (1).} We need to prove that
for any $(t_0,x_0)\in I\times\R^2$ and $\epsilon>0$
there is a neighbourhood $U$ of $(t_0,x_0)$ such that
$\abs{h(t,x)-h(t_0,x_0)}<\epsilon$ for any $(t,x)\in U$.
Set $\epsilon>0$. If $t_0\neq0$ the claim is obvious.
Let $t_0=0$ and $x_0\in\R^2$. We show that there are
$\rho>0$ and $\delta>0$ such that 
$\abs{h(t,x)-h(0,x_0)}<\epsilon$ for 
$t<\rho$ and $\abs{x-x_0}<\delta$.
Since for $t=0$ we have 
$\abs{h(0,x)-h(0,x_0)}=\abs{Df_0(x-x_0)}$,
we can assume that $t\neq0$.
By the differentiability of $f$ at $x=0$, we have
$\lim_{\abs{x}\to0}\frac{f(x)-Df_0(x)}{\abs{x}}=0$.
Observe that for sufficiently small both $t$ and $\abs{x-x_0}$
we have
\begin{multline*}
\abs{h(t,x)-h(0,x_0)}=
\abs{\frac{f(tx)}{t}-Df_0(tx)}\\
\le\abs{\frac{f(tx)-Df_0(tx)}{\abs{tx}}}\abs{x}
+\abs{Df_0(x)-Df_0(x_0)}\le\frac\epsilon2+\frac\epsilon2=\epsilon,
\end{multline*}
which is our claim.\vspace{1mm}

\noindent \emph{Ad (2).} It is enough to show 
that for every $m>0$ there is $l>0$ such that 
$\abs{h(t,x)}>m$ for all $t\in I$ and $\abs{x}>l$.
Observe that
\begin{enumerate}[(a)]
	\item there is $\epsilon>0$ such that 
	$\abs{Df_0(x)}\ge\epsilon\abs{x}$ for any $x\in\R^2$,
	\item there is $\delta_1>0$ such that 
	$\abs{f(x)-Df_0(x)}\le\frac\epsilon2\abs{x}$
	for any $\abs{x}<\delta_1$,
	\item from (a) and (b) for any $\abs{x}<\delta_1$ we have
	$\abs{f(x)-Df_0(x)}\le\frac12\abs{Df_0(x)}$ and,
	in consequence, $\abs{f(x)}\ge\frac12\abs{Df_0(x)}$,
	\item there is $m_1\in(0,m)$ such that $\abs{f(x)}>m_1$
	for any $\abs{x}\ge\delta_1$,
	\item	there is $\delta_2>0$ such that $\abs{f(x)}>m$
	for any $\abs{x}\ge\delta_2$.
\end{enumerate}
Set $t_1:=m_1/m$ and $l=:\max{\{\delta_2/t_1,2m/\epsilon\}}$.
Let $\abs{x}>l$. Now we only need to consider 
the following three cases.\vspace{1mm}

\noindent Case $t_1\le t\le1$. Since $\abs{tx}>\delta_2$,
we get $\abs{h(t,x)}\ge\abs{f(tx)}>m$, by (e).\vspace{1mm}

\noindent Case $0<t<t_1$. If $\abs{tx}\ge\delta_1$ then
$\abs{h(t,x)}>\frac{m_1}{t_1}=m$ 
from (d) and the definition of $t_1$.
If $\abs{tx}<\delta_1$ then
\[
\abs{h(t,x)}=\abs{\frac{f(tx)}{t}}\ge
\frac12\abs{\frac{Df_0(tx)}{t}}=\frac12\abs{Df_0(x)}>m
\]
from (a), (c) and the definition of $l$.\vspace{1mm}

\noindent Case $t=0$. We have 
$\abs{h(0,x)}=\abs{Df_0(x)}>2m>m$ from (a)
and the definition of $l$.
\end{proof}

\begin{rem}
Lemma~\ref{lem:proof2} is true for maps $f\colon\R^n\to\R^k$,
but the assumption that $Df_0$ is nonsingular must be replaced
with $\rank Df_0=n$.
\end{rem}

\begin{cor}\label{cor:proof2}
Let $f=\nabla\varphi$ be generic
and $f^{-1}(0)=\{p\}$. Then $f$ is proper gradient homotopic to 
the Hessian map $\Hess_p\!\varphi\colon\R^2\to\R^2$. 
\end{cor} 

\subsection{Raising and lowering critical points}
\begin{lem}\label{lem:bump}
Let $\varphi\colon U\subset\R^2\to\R$ be a $C^2$ function
such that $\nabla\varphi$ is generic 
and $0$ be a saddle point of $\varphi$. Then there are 
a neighbourhood $U'\subset U$ of $0$ such that 
$\cl U'\subset U$ and a Morse function
$\psi\colon U\subset\R^2\to\R$ such that
\begin{itemize}
	\item $0$ is a saddle and the only critical point of $\psi$ in $U'$,
	\item $\restrictionmap{\psi}{U\setminus\cl U'}=
	\restrictionmap{\varphi}{U\setminus\cl U'}$,
	\item $\psi(0)<\varphi(0)$.
\end{itemize}
\end{lem}
 
\begin{proof}
Choose a sufficiently small disc around $0$ and define a bump function $\mu$
centered at $0$ with compact support contained in that disc.
The function $\psi=\varphi-\mu$ has required properties
(compare~\cite[Thm 2.34]{M}).
\end{proof}

\subsection{Gradient fields on curvelinear quadrangles}
In the next lemmas $A=x_1x_2x_3x_4$ denotes
a smooth curvelinear quadrangle in the plane
with right angles at the corners
and $B=X_1X_2X_3X_4=I\times I$ (the unit square).
The following result may be viewed as a version of
the Schoenflies theorem for such quadrangles. 
\begin{lem}\label{lem:square}
There is a diffeomorphism
$\theta\colon B\to A$ such that $\theta(X_i)=x_i$
for $i=1,2,3,4$. 
\end{lem}

\begin{rem}\label{rem:field}
The assertion of the above lemma can be 
strengthened in the following way.
Let $\gamma'\colon I\to\arc x_1x_4$
and $\gamma''\colon I\to\arc x_2x_3$
be any smooth parametrizations.
We can guarantee that the diffeomorphism $\theta$
satisfies for all $s\in I$ the following conditions:
\begin{itemize}
	\item $\theta(0,s)=\gamma'(s)$
	and $\theta(1,s)=\gamma''(s)$,
	\item  the curves $\theta(s\times I)$ and $\theta(I\times s)$
	are perpendicular to the respective sides of $A$.
\end{itemize}
\end{rem}

\begin{lem}\label{lem:nonvan}
Assume that $w\colon\partial B\to\R^2$ is a continuous nonvanishing
vector field such that
\begin{itemize}
	\item $w(0,y)=w(1,y)=(0,1)$ for all $y\in I$,
	\item $w(x,0)=(0,w'(x))$ and $w(x,1)=(0,w''(x))$, 
	where $w',w''\colon$ $I\to\R$ are $C^1$ functions.
\end{itemize}
Then there is a $C^1$ function $\psi\colon B\to\R$
such that $\nabla\psi$ is nonvanishing on $B$
and $\restrictionmap{\nabla\psi}{\partial B}=w$.
\end{lem}

\begin{proof}
Let $m:=\max{\{w'(x),w''(x)\mid x\in I\}}$, $a(x):=x(1-x)$ and 
$\mu\colon I\to\R$ be a continuous function such that
$\mu(0)=1$, $\mu(1)=0$, $-1/m\le\mu\le1$ and $\int_0^1\mu(t)dt=0$. 
Set $C:=\{(x,y)\in B\mid y\le a(x)\}$. Let us define a function
$u\colon C\to\R$ by
\[
u(x,y)=\begin{cases}
0& \text{if $x\in\{0,1\}$},\\
\int_0^y\mu(t/a(x))dt& \text{otherwise},
\end{cases}
\]
and a function $\psi\colon B\to\R$ by 
\[
\psi(x,y)=y+\begin{cases}
(w'(x)-1)u(x,y)& \text{if $y\le a(x)$},\\
0& \text{if $a(x)<y<1-a(x)$},\\
(1-w''(x))u(x,1-y)& \text{if $1-a(x)\le y$}.
\end{cases}
\]
One can check that $\psi$ satisfies
the assertion of our lemma.
In particular, $\nabla\psi$ is nonvanishing,
because $\frac{\partial\psi}{\partial y}>0$ for $(x,y)\in B$.
\end{proof}

\begin{lem}\label{lem:field}
Let $v\colon\partial A\to\R^2$ be
a continuous nonvanishing vector field such that
\begin{itemize}
	\item $\restrictionmap{v}{\arc x_1x_2}$ and $\restrictionmap{v}{\arc x_3x_4}$
	are perpendicular to $\partial A$ and smooth,
	\item $\restrictionmap{v}{\arc x_1x_4}$ and $\restrictionmap{v}{\arc x_2x_3}$
	are tangent to $\partial A$,
	\item $\int_{\arc x_1x_4}\!\abs{v} ds=\int_{\arc x_2x_3}\!\abs{v} ds=1$.
\end{itemize}
Then there is a gradient nonvanishing extension
of $v$ to $A$.
\end{lem}

\begin{proof}
First we uniquely fix parametrizations
$\gamma'\colon I\to\arc x_1x_4$ and
$\gamma''\colon I\to\arc x_2x_3$ so that for each $c\in I$
we have
\[
\int\limits_{\arc x_1\gamma'(c)}\!\!\!\!\!\!\!\!\!\abs{v} ds=
\!\!\!\!\!\!\!\!\!
\int\limits_{\arc x_2\gamma''(c)}\!\!\!\!\!\!\!\!\!\abs{v} ds=c.
\]
Let $\theta\colon B\to A$ be a diffeomorphism
satisfying the conditions mentioned in Lemma~\ref{lem:square} 
and Remark~\ref{rem:field}. Without loss of generality we may assume
that $v(x_1)$ agrees with the orientation of $\arc x_1x_4$.

Define $w\colon\partial B\to\R^2$ by $w(z)=(D\theta(z))^T\cdot v(\theta(z))$.
Note that $w$ satisfies the assumptions of Lemma~\ref{lem:nonvan}.
Hence there is a $C^1$ function $\psi\colon B\to\R$
such that $\nabla\psi$ is nonvanishing on $B$
and $\restrictionmap{\nabla\psi}{\partial B}=w$.

Finally, define $\varphi\colon A\to\R$ by 
$\varphi(z)=\psi(\theta^{-1}(z))$.
Obviously $\nabla\varphi$ is nonvanishing on $A$. 
Moreover, for $z\in\partial A$
\[
\nabla\varphi(z)=
(D\theta^{-1}(z))^T\cdot\nabla\psi(\theta^{-1}(z))=
(D\theta^{-1}(z))^T\cdot w(\theta^{-1}(z))=v(z),
\]
which completes the proof.
\end{proof}


\begin{thebibliography}{9}

\bibitem{BP1}
{P. Bart{\l}omiejczyk, P.Nowak-Przygodzki}, 
\emph{Gradient otopies of gradient local maps}, 
{Fund. Math.} 
\textbf{214} (2011), 
89--100.

\bibitem{BP2}
{P. Bart{\l}omiejczyk, P. Nowak-Przygodzki}, 
\emph{Proper gradient otopies}, 
{Topol. Appl.} 
\textbf{159} (2012), 
2570--2579.

\bibitem{BG}
J. C. Becker, D. H. Gottlieb, 
\emph{Vector fields and transfers},
Manuscripta Math. \textbf{72} (1991),
111--130.

\bibitem{D} 
E. N. Dancer,
\emph{A new degree for S1-invariant gradient mappings and applications},
Ann. Inst. Henri Poincar\'e, Analyse Non Lin\'eaire, 
\textbf{2} (1985), 329--370.

\bibitem{L}
F. Laudenbach, 
\emph{A proof of Morse’s theorem about the cancellation of critical points}, 
C. R. Math. Acad. Sci. Paris \textbf{351} (2013), 483--488. 

\bibitem{H}
M. W. Hirsch,  
\emph{Differential Topology}, 
Springer, New York, 
1976.

\bibitem{M}
Y. Matsumoto,
\emph{An Introduction to Morse Theory}, 
Translations of Mathematical Monographs vol. 208, 
American Math. Soc. (2002). 

\bibitem{P}
A. Parusi\'nski, 
\emph{Gradient homotopies of gradient vector fields}, 
{Studia Math.} \textbf{XCVI} (1990), 73--80.

\bibitem{S}
M. Starostka, 
\emph{Connected components of the space 
of proper gradient vector fields},
\texttt{arXiv:1809.01411 [math.DS]} 

\end{thebibliography}
\end{document}